\documentclass{article}

\usepackage{CAAI_preprint}

\usepackage{amsrefs}
\title{\textbf{Modelling and analysis of a modified May--Holling--Tanner predator-prey model with Allee effect in the prey and an alternative food source for the predator}}

\author{
Claudio Arancibia--Ibarra \\
School of Mathematical Sciences, Queensland University of Technology\\
GPO Box 2434, GP Campus, Brisbane, Queensland 4001 Australia\\
Facultad de Ingeniería y Negocios, Universidad de Las Américas \\
Av. Manuel Montt 948, Santiago, Chile \\
 \texttt{claudio.arancibiaibarra@qut.edu.au} \\
  \AND
Jos\'e Flores \\
Department of Mathematics, The University of South Dakota (USD)\\
Vermillion, South Dakota, USA\\
 \texttt{Jose.Flores@usd.edu} \\  
}

\begin{document}
\maketitle
\begin{abstract} \normalsize

In the present study, we have modified the traditional May--Holling--Tanner predator-prey model used to represent the interaction between least weasel and field-vole population by adding an Allee effect (strong and weak) on the field-vole population and alternative food source for the weasel population. It is shown that the dynamic is different from the original May--Holling--Tanner predator-prey interaction since new equilibrium points have appeared in the first quadrant. Moreover, the modified model allows the extinction of both species when the Allee effect (strong and weak) on the prey is included, while the inclusion of the alternative food source for the predator shows that the system can support the coexistence of the populations, extinction of the prey and coexistence and oscillation of the populations at the same time. Furthermore, we use numerical simulations to illustrate the impact that changing the predation rate and the predator intrinsic growth rate have on the basin of attraction of the stable equilibrium point or stable limit cycle in the first quadrant. These simulations show the stabilisation of predator and prey populations and/or the oscillation of these two species over time.

\end{abstract}


\section{Introduction}

Numerous investigations have extensively studied the interactions between the specialist predator least-weasel (\textit{Mustela nivalis}) and the field-vole (\textit{Microtus agrestis})~\cite{hanski2,hanski,turchin2}. These investigations have revealed important findings in the variation of the populations in Fennoscandia from regular cycles to stable populations. This variation is also affected by the fact that the rodent breeding is seasonal, and the duration of the breeding varies with the physical location in Fennoscandia~\cite{dalkvist}. Moreover, cycles may be affected by the existence of an optimum group size and an Allee effect~\cite{madison}. Graham and Lambin~\cite{graham} showed that field-vole survival depends on the variation of weasel predation. They also demonstrated that the proportion of weasels was suppressed in summer and autumn, while the vole population always declined to a low density. Thus, if the predation by weasels is reduced then the field-vole survival increases. In~\cite{graham}, the authors assumed that the May--Holling--Tanner model was not appropriate to conduct such studies due to the large number of parameters. Additionally, the authors suggested that there are different variations in these two species that should be considered in order to improve study the results.   

The most common mathematical framework for describing this predator-prey interaction is that the weasel and field-vole have a logistic-type growth function and the predator carrying capacity is prey dependent~\cite{turchin}. The classic model used to represent the predation between these two species is the May--Holling--Tanner predator-prey model~\cite{may} which is a modification of original the Leslie--Gower model. Leslie (1948) proposed the following generalisation of the Lotka-Volterra model
\begin{equation}\label{LG}
\begin{aligned}
\dfrac{dN}{dt}&=rN\left(1-\dfrac{N}{K}\right)-qNP, \\
\dfrac{dP}{dt} &= sP\left(1-\dfrac{P}{nN}\right).
\end{aligned}
\end{equation}
where $N(t)$ is used to represent the size of the prey population at time $t$, $P(t)$ is used to represent the size of the predator population at time $t$, $r$ is the intrinsic growth rate for the prey, $s$ is the intrinsic growth rate for the predator, $K$ is the prey environmental carrying capacity, $n$ is a measure of the quality of the prey as food for the predator, $nN$ can be interpreted as a prey dependent carrying capacity for the predator and the term $P/nN$ is known as the Leslie-Gower term. It measures the loss in the predator population due to rarity (per capita $P/N$) of its favourite food. 

This model retains the prey equation from the Volterra model, but for the predator it assummes a logistic-like term, in which the predator carrying capacity is directly proportional to prey density. This interesting formulation for the predator dynamics has been discussed in Leslie and Gower in~\cite{leslie} and Pielou in~\cite{pielou}. Later, May (1974) in~\cite{may} modified the Leslie-Gower model~\eqref{LG} by assuming the hyperbolic functional response (Holling type II) in the prey equation, $H(N)=qN/(N+a)$, which occurs in species when the number of prey consumed rises rapidly at the same time that the prey density increases~\cite{turchin,may,santos}. The parameter $q$ is the maximum predation rate per capita and $a$ is half of the saturated level~\cite{turchin}. In particiular, it is given by
\begin{equation}\label{eq1}
\begin{aligned}
\dfrac{dN}{dt}&=rN\left(1-\dfrac{N}{K}\right)-q\dfrac{NP}{N+a}. \\
\dfrac{dP}{dt} &= sP\left(1-\dfrac{P}{nN}\right).
\end{aligned}
\end{equation}
Model~\eqref{eq1} is known as Holling--Tanner or as May--Holling--Tanner model which is described by an autonomous two-dimensional system of ordinary differential equations~\cite{turchin,aguirre1,arancibia7}.

The original May--Holling--Tanner model~\eqref{eq1} was studied, among others, by~\cite{saez}. The authors proved that all the solutions which are initiated in $\mathbb{R}^2_+$ are bounded and the equilibrium point $(1,0)$, which is the scaled equilibrium point $(K,0)$, is a saddle point. Moreover, by doing a polar blowing-up~\cite{dumortier} at the origin, Saez and Gonz\'alez-Olivares~\cite{saez} showed that there exists a separatrix curve between a hyperbolic and a parabolic sector in the neighbourhood of the origin. The authors also found that there is always one positive equilibrium point in the first quadrant which can be stable; unstable surrounded by a stable limit cycle; or stable surrounded by two limit cycles. Additionally, the global stability of a periodic solution of this predator-prey model was investigated by~\cite{lisena} and the analysis of a discrete May--Holling--Tanner model was described in~\cite{cao}. 
\begin{figure}
\begin{center}
\includegraphics[width=11.5cm]{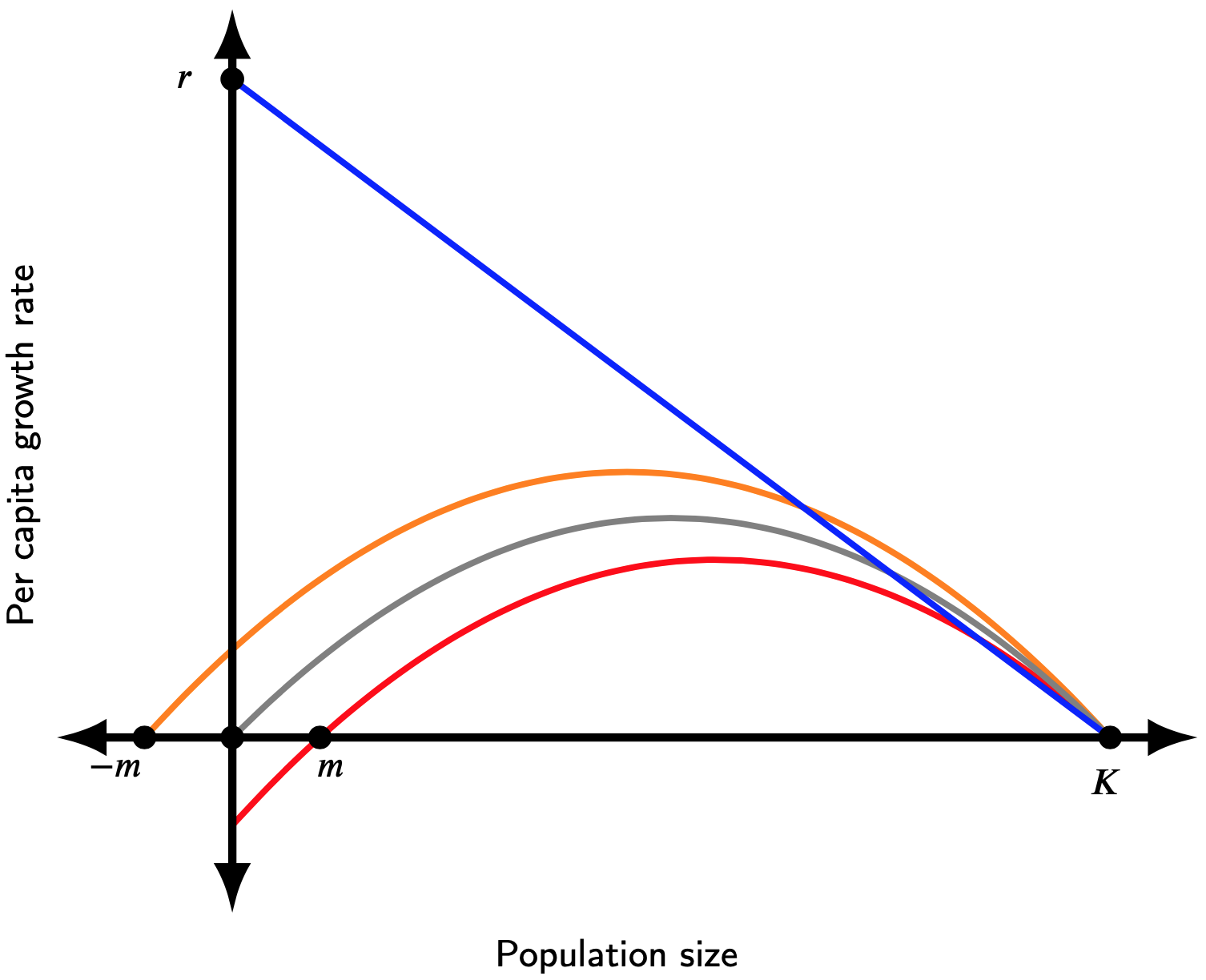}
\end{center}
\caption{We show the per capita growth rate of the logistic function (blue line), the strong Allee effect with $m=15$ (red curve), and the weak Allee effect with $m=-15$ (orange curve) and $m=0$ (grey curve).}
\label{FIG01}
\end{figure} 

Additional complexity can be incorporated into the predator-prey model~\eqref{eq1} in order to make it more realistic. We can consider, in the prey population, a mechanism connected with a density-dependent phenomenon (Allee effect~\cite{allee2,allee,berec,buffoni,van,wang7}). The Allee effect may appear due to a wide range of biological phenomena such as reduced anti-predator vigilance; social thermo-regulation; genetic drift; mating difficulty; reduced defence against the predator; and deficient feeding because of low population densities~\cite{stephens2}. The Allee effect may appear in the field-vole population due to the density-dependent habitat selection showed in~\cite{jankovic,morris}. Moreover, the stabilisation of field-vole populations depends on the variation in reproductive rate and recruitment of the population~\cite{ostfeld}. The influence of the Allee effect upon the logistic-type growth in the prey equation is represented by the inclusion of a multiplier in the form of $N-m$~\eqref{eq4}. For $0<m<K$, the per-capita growth rate of the prey population with the Allee effect included is negative, but increasing, for $N\in[0,m)$, and this is referred to as the “strong Allee effect”. When $m\leq0$\footnote{Note that $m=0$ is often also called the “weak Allee effect”, however, the behaviour of the model with $m=0$ is similar to the case of a strong Allee effect ($m>0$) and thus we will not consider the case of $m=0$ in this manuscript.}, the per-capita growth rate is positive but increases at low prey population densities and this is referred to as the weak Allee effect~\cite{berec,courchamp2}. The per capita growth rate for the logistic growth function and Allee effects (strong and weak) are shown in Figure~\ref{FIG01}. We observe that the per capita growth rate for a weak Allee effect is positive and growing, when compared to the strong Allee effect. Moreover, the trajectories of a population affected by a weak Allee effect are delayed compared with a strong Allee effect. This effect can be generated by the reduction of the male population~\cite{liermann} or the habitat selection~\cite{jankovic,morris}.

\begin{table}
\begin{center}
\caption{In the first row the prey growth rate is affected by Allee effect and in the second row the predator growth rate is affected by an alternative food source.}
\begin{tabular}{ccc} \hline
Modified growth equations 	&  Modified models		\\\hline
\\
Allee effect: &\\ 
\begin{minipage}{7cm}\begin{equation}\label{eq4} F_{m}(N)=r\left(1-\dfrac{N}{K}\right)(N-m)\end{equation}\end{minipage} & 
\begin{minipage}{8cm}\begin{equation}\label{eq5}\begin{aligned}
\dfrac{dN}{dt} &= rN\left(1-\dfrac{N}{K}\right)\left(N-m\right)-\dfrac{qNP}{N+a}, \\
\dfrac{dP}{dt} &= sP\left(1-\dfrac{P}{nN}\right).
\end{aligned}\end{equation}\end{minipage}	\\
\\
Alternative food: &\\  
\begin{minipage}{7cm}\begin{equation}\label{eq6}G_{m}(N,P)=s\left(1-\dfrac{P}{nN+c}\right)\end{equation}\end{minipage}& 
\begin{minipage}{8cm}\begin{equation}\label{eq7}\begin{aligned}
\dfrac{dN}{dt} &= rN\left(1-\dfrac{N}{K}\right)-\dfrac{qNP}{N+a}, ~~~~~~~~~~\\
\dfrac{dP}{dt} &= sP\left(1-\dfrac{P}{nN+c}\right).
\end{aligned}\end{equation}\end{minipage} \\
\\
Allee effect and alternative food: &\\ 
\begin{minipage}{7cm}equations~\eqref{eq4} and~\eqref{eq6}\end{minipage}& 
\begin{minipage}{8cm}\begin{equation}\label{eq8}\begin{aligned}
\dfrac{dN}{dt} &= rN\left(1-\dfrac{N}{K}\right)\left(N-m\right)-\dfrac{qNP}{N+a}, \\
\dfrac{dP}{dt} &= sP\left(1-\dfrac{P}{nN+c}\right).
\end{aligned}\end{equation}\end{minipage} \\
\\\hline
\end{tabular}
\end{center}
\end{table}

Some species are considered as a generalist and specialist predator and thus it can switch to another available food, although its population growth may still be limited by the fact that its preferred food is not available in abundance. For instance, McDonald {\em et al.\ }~\cite{mcdonald} reveal that weasels' diet involved mainly field-vole, which correspond to $68\%$  of the diet which can be considered as the primary food source. Moreover, the authors also reveal that lagomorphs with $25\%$ of the diet; birds and birds' eggs with $5\%$ of the diet and other types of food with $2\%$ of the diet can be considered as a secondary food source, see Figure~\ref{FIG02}. This characteristic was not considered in investigations associated with the interaction between the field-vole as a prey and least-weasel as a predator~\cite{hanski2,hanski,turchin,wollkind,hanski3,andersson,erlinge,hansson}. That is, these studies do not consider that, since the predator is a generalist, it can survive under different environments and utilise a large range of food resources. Instead of adding more species to the model, we assume that these other food sources are abundantly available~\cite{mcdonald} and model this characteristic by adding a positive constant $c$ to the environmental carrying capacity for the predator~\cite{aziz}. Therefore, we have made a modification to the prey-dependent logistic growth term in the predator equation, namely $nN$ is replaced by $nN+c$~\eqref{eq6}. As a result, we observe that the modified May--Holling--Tanner model with an alternative food source~\eqref{eq7} is no longer singular for $N=0$.    

\begin{figure}
\begin{center}
\includegraphics[width=14cm]{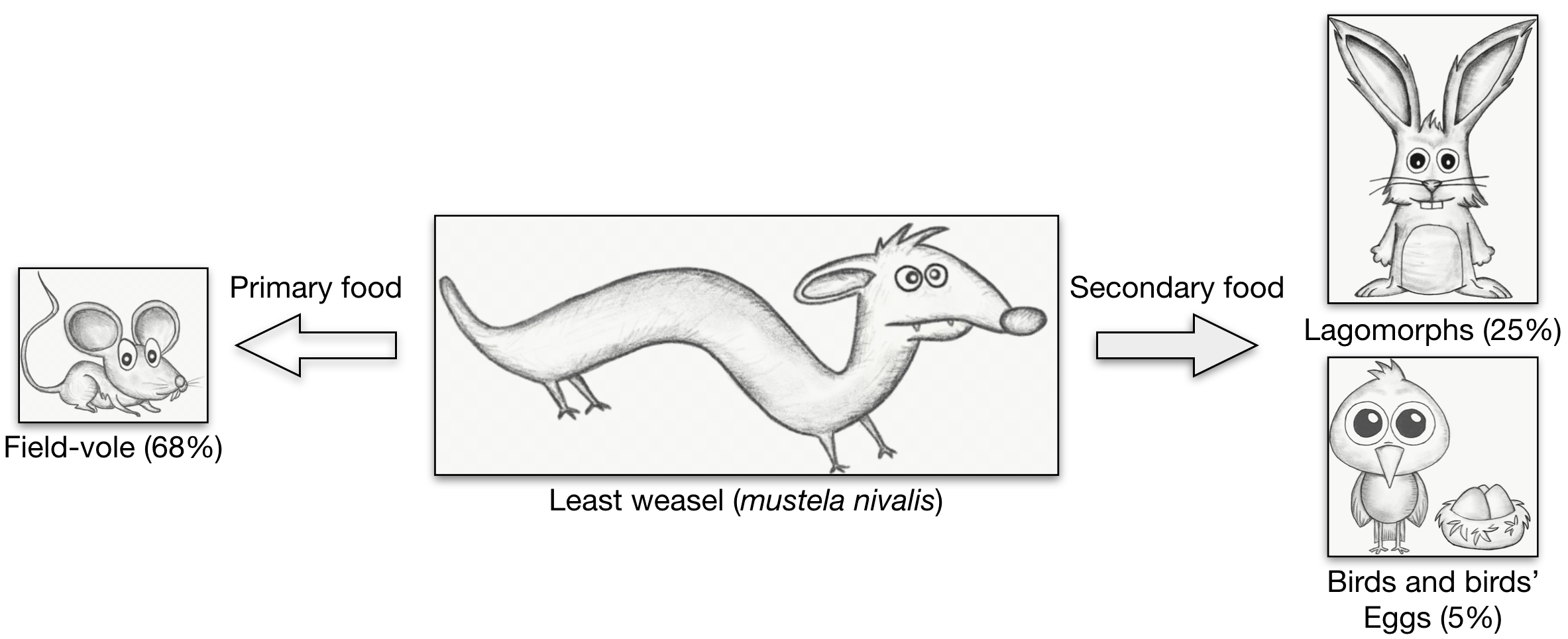}
\end{center}
\caption{Main and secondary diet of a specialist predator least weasel (\textit{mustela nivalis}) studied in~\cite{mcdonald}.}
\label{FIG02}
\end{figure}  

The aim of this manuscript is to extends the properties of the weasel and field-vole interaction studied in~\cite{graham} by showing the impact of an Allee effect on the prey~\eqref{eq5} and an alternative food source for the predator~\eqref{eq7} in the original May--Holling--Tanner model. In addition, the analysis of the modified model with weak Allee effect (i.e. systems~\eqref{eq5} with $m<0$) and the modified model considering an alternative food for the predator (i.e. system~\eqref{eq7}) show a strong connection with the findings presented in~\cite{turchin2,hanski3}. Study results suggest that the system does not lose the coexistence and oscillation of the population when a constant alternative food source for weasels and a weak Allee effect on the prey are included. Moreover, the survival of the field-vole does not depend on the level of predation when a weak Allee on the prey is included~\cite{graham}. However, we have shown that the May--Holling--Tanner model with strong Allee effect on the prey (i.e. systems~\eqref{eq5} with $m>0$) reduce the complexity of the dynamics in the modified model. The modified model with strong Allee effect supports, depending on the initial conditions, extinction and coexistence of both the prey and the predator species. 

The mathematical interaction in systems~\eqref{eq1}, considering the modifications separately and together, is study further in Section~\ref{MHT}. In Section~\ref{co}, we conclude the manuscript, presenting the impact of the modification on the dynamic of the modified May--Holling--Tanner models. We also discuss the ecological implications of the inclusion of the strong and weak Allee effects on the prey and of an alternative food for the predator in a real model studied in~\cite{turchin2,graham,hanski3}.

\section{Models}\label{MHT}

The results derived in~\cite{saez} are for the nondimensionalised versions of the modified versions of systems~\eqref{eq1} which is given by
\begin{equation*}
\begin{aligned}
\dfrac{dN}{dt} & = rN\left(1-\dfrac{N}{K}\right)-\dfrac{qNP}{N+a}, \\ 
\dfrac{dP}{dt} & = sP\left(1-\dfrac{P}{nN}\right).
\end{aligned}  
\end{equation*}

However, this nondimensionalised version is topologically equivalent to the original version of the model as we used a diffeomorphism preserving the orientation of time. Therefore, the results can be transferred to the original May--Holling--Tanner model. Moreover, we can extend the analysis by using the system parameters presented in~\cite{turchin2}. In this paper, Turchin and Hanski used the original May--Holling--Tanner model and the authors showed that the oscillation of the field-vole is driven by their interaction with the weasel. This oscillation is obtained by using the system parameters presented in Table~\ref{T02}.

\begin{table}
\begin{center}
\caption{Description of the system parameters, their units and values that will be used in this manuscript based in Turchin \& Hanski~\cite{turchin2} results. We add an alternative food for the predator and the Allee threshold values.} 
\label{T02}
\begin{tabular}{ccccc} 
\hline
Description of parameter						& Symbols 	& Values		& Units						\\\hline
Intrinsic growth rate of field-vole population	& $r$		& $4$		& yr$^{-1}$ 					\\
Field-vole carrying capacity					& $K$		& $150$		& voles ha$^{-1}$
				\\
Maximum consumption per weasel	 			& $q$		& $500-700$	& voles yr$^{-1}$ weasel$^{-1}$ 	\\
Half of the saturated level					& $a$		& $6$		& voles ha$^{-1}$
			\\
Intrinsic growth rate of weasel population		& $s$		& $1-1.5$		& yr$^{-1}$					\\
Quality of the field-vole as food for the weasel	 & $n$		& $0.025$		& voles $^{-1}$ weasel 			\\
Quality of alternative species as food for the weasel	& $c$		& $0.01$		& weasel 			\\
Allee threshold					& $m$		& $\pm15$		& voles ha$^{-1}$ 			\\
\hline
\end{tabular}
\end{center}
\end{table}

Nowadays, there are several computational methods to find different type of bifurcations. These methods are implemented in software packages such as MATCONT~\cite{matcont2}. Figure~\ref{FIG05} illustrates the Hopf bifurcation which was detected with MATCONT for the parameters $r$, $K$, $a$, $n$ and $c$ fixed, while we will vary $q$ and $s$, see Table~\ref{T02}.

\subsection{Original model}
The original May--Holling--Tanner predator-prey model with functional response Holling type II~\eqref{eq1} was studied mathematically, for instance, in~\cite{saez,huang} and used to model real predator-prey interactions~\cite{hanski2,hanski,roux,turchin}. The bifurcation diagram of the original May--Holling--Tanner model~\eqref{eq1} created with MATCONT~\cite{matcont2} and two typical phase portraits are shown in the first row of Figure~\ref{FIG05}. The Hopf curve\footnote{The Hopf curve is where the stability of an equilibrium point switches and, in some cases, a periodic solution arises.} for the May--Holling--Tanner model shows that if the predation rate $q$ and the intrinsic growth of the predator $s$ are located in the hatched green area, then system~\eqref{eq1} supports the stabilisation of the predator and the prey populations. If the predation rate is located in the blue area, then there are conditions for which the system supports the oscillation of both populations, i.e., there is a stable limit cycle surrounding a positive equilibrium point in the first quadrant. Moreover, there are conditions in the $(q,s)$ parameter space (near the Hopf curve) for which populations can coexist and also oscillate simultaneously, i.e., there are two limit cycles surrounding a positive equilibrium point in the first quadrant. However, for the sake of the presentation, this region is not included in the bifurcation diagram. Instead, we refer to~\cite{saez}. Note that the extinction of (one of) the species is not possible in this model. 

\subsection{Allee effect (strong and weak)}
The May--Holling--Tanner model with a strong Allee effect ($m>0$) was studied in~\cite{arancibia3}. The authors proved that system~\eqref{eq5} with $m>0$ has three equilibrium points in the $x$-axis and at most two positive equilibrium points, see the left panel of Figure~\ref{FIG03}. The authors also proved that one of the positive equilibrium points is always a saddle. While, the other can be stable, unstable or stable surrounded by an unstable limit cycle. For $m<0$ there are two equilibrium points on the $x$-axis and at most three positive equilibrium points, see the right panel of Figure~\ref{FIG03}. The stability of the equilibrium points of system~\eqref{eq5} considering $m<0$ is studied partially in Appendix~A. 

First, we recall that as we used a diffeomorphism preserving the orientation of time the nondimensionalised version~\eqref{eq12} is topologically equivalent to the original version of the model~\eqref{eq5}. Therefore, the results can be transferred to the modified May--Holling--Tanner model. Moreover, we can extend the analysis by using the system parameters $(r,K,a,n,c,m)=(4,150,6,0.025,0.01,\pm15)$ (see Table ~\ref{T02}). The bifurcation diagram of the May--Holling--Tanner model with weak Allee effect~\eqref{eq5} ($m<0$) and two typical phase portraits are shown in the third row of Figure~\ref{FIG05}. System~\eqref{eq5} with $m<0$ is not discussed in detail in this manuscript, but the bifurcation diagram can be obtained by using the techniques presented in~\cite{arancibia3}. If the parameters $(q,s)$ are located in the hatched green region, then system~\eqref{eq5} has one positive equilibrium point which is stable, and thus both the predator and the prey population can coexist. If the parameters $(q,s)$ are located in the blue region, then the equilibrium point is unstable surrounded by a stable limit cycle and thus both the predator and the prey populations oscillate. 

However, if the parameters $(q,s)$ are located in the grey region, then system~\eqref{eq5} has three equilibrium points. These equilibrium points can be a saddle, a stable and/or an unstable node. Therefore, system~\eqref{eq5} supports the coexistence and/or oscillation of the predator and prey populations. Note that the dynamics of system~\eqref{eq5} with $m<0$, outside the grey region, are qualitatively similar to the dynamics of system~\eqref{eq1} as shown in row one in Figure~\ref{FIG05} since there is a large region in the $(q,s)$ parameter space in which there is always one positive equilibrium point, i.e. the hatched green region. 

Additionally, the Allee effect, both strong and weak, decreases the range of the value of the predation rate for which both populations can coexist, i.e. $\tilde{q}_2<q_1$ (see Figure~\ref{FIG05}). Moreover, as in the original May--Holling--Tanner model, the predator and prey populations never become extinct in system~\eqref{eq5} and thus there are always conditions in the system parameters where both populations coexist (see hatched green area in Figure~\ref{FIG05}) and/or oscillate (see blue and grey areas in Figure~\ref{FIG05}). 

The bifurcation diagram of the modified May--Holling--Tanner model with strong Allee effect~\eqref{eq5} ($m>0$) studied in~\cite{arancibia3} and two typical phase portraits are shown in the fourth row of Figure~\ref{FIG05}. The Hopf curve for the model shows that if $(q,s)$ are located in the green area, system~\eqref{eq5} supports the stabilisation of both populations and also the extinction of both populations. However, if the predation rate is located in the red hatched area in Figure~\ref{FIG05} ($q>\tilde{q}_1$), then system~\eqref{eq5} does not have positive equilibrium points in the first quadrant and thus the system only supports the extinction of both the predator and the prey populations. In contrast, if the predation rate and the intrinsic growth of the predator $(q,s)$ are located in the solid red area, then system~\eqref{eq5} has two positive equilibrium points which are a saddle and an unstable node, and thus again we observe the extinction of the predator and prey populations.  

\subsection{Alternative food}
The model with alternative food for predators corresponds to system~\eqref{eq7}. This model was studied in~\cite{arancibia7, arancibia2, arancibia5} where the authors presented the stability of the equilibrium points and conditions for periodic solutions. The bifurcation diagram of the modified May--Holling--Tanner model~\eqref{eq7} and two typical phase portraits are shown in the second row of Figure~\ref{FIG05}, see also~\cite{arancibia9,arancibia5}. The Hopf curve for the modified May--Holling--Tanner model shows that if the predation rate $q$ and the intrinsic growth of the predator $s$ are located in the green area, then system~\eqref{eq7} supports the stabilisation of both populations and also the extinction of only the prey population. If the predation rate is located in the hatched brown area in Figure~\ref{FIG05} ($q>q_2$), then system~\eqref{eq7} does not have a positive equilibrium point in the first quadrant and thus the extinction of the prey population is observed, while the predator population stabilises at $(0,c)$ (the equilibrium point $(0,c)$ is related to the alternative food source for the predator). 

Moreover, when the parameters $(q,s)$ are located in the solid brown area, then system~\eqref{eq7} has two positive equilibrium points in the first quadrant, namely a saddle and an unstable node, and the extinction of the prey population is again observed for all initial conditions. When the parameters $(q,s)$ are located in the hatched brown area, then system~\eqref{eq7} does not have equilibrium points in the first quadrant and the extinction of the prey population is observed for all initial conditions. Additionally, in~\cite{arancibia9,arancibia5} the authors have shown that there is a region, near to the Hopf bifurcation, in which both populations oscillate and coexist at the same time, i.e. there are two limit cycles surrounding a positive equilibrium point in the first quadrant. For the sake of the presentation, we did not include this region in the bifurcation diagram. 

\subsection{Allee effect (strong and weak) and alternative food}
The May--Holling--Tanner model with Allee effect on the prey and alternative food for predators is given by system~\eqref{eq8}. This model was studied partially in~\cite{arancibia} and in Appendix~B. The bifurcation diagram of the modified May--Holling--Tanner model with Allee effect on the prey and alternative food for the predator~\eqref{eq8} and two typical phase portraits are shown in the last row of Figure~\ref{FIG05}. System~\eqref{eq1} with Allee effect is not discussed in detail in this manuscript, but the bifurcation diagram can also be obtained by using the techniques presented in~\cite{arancibia3}. The Hopf curve for the system~\eqref{eq8} with Allee effect and alternative food shows that system~\eqref{eq8} supports the coexistence of both populations, the extinction of the prey population only (green area) and the extinction of the prey population (solid and hatched brown regions). That is, the behaviour is similar to the modified May--Holling--Tanner model~\eqref{eq7}, see the second row of Figure~\ref{FIG05}.

In Figure~\ref{FIG05}, we show the bifurcation diagrams (first columns) in which different regions represent different behaviours. In the solid green region, sets of initial conditions exist that lead to the coexistence of both species, while there are also open sets of initial conditions that lead to the extinction of both species. In the hatched green regions, all initial conditions lead to the coexistence of both populations. In the solid blue regions, all initial conditions lead to the oscillation of both populations. In the solid and hatched red regions, all initial conditions lead to the extinction of both species, while in the solid and hatched brown regions, all initial conditions lead to the extinction of the prey only. In the right column, the orange region (dark and light) represents the basin of attraction of the stabilisation of both populations, the light blue (light grey) region represents the basin of attraction of the extinction of both (only the prey) populations and the yellow region represent the oscillation of both populations.

\begin{figure}
\centering
\includegraphics[width=14cm]{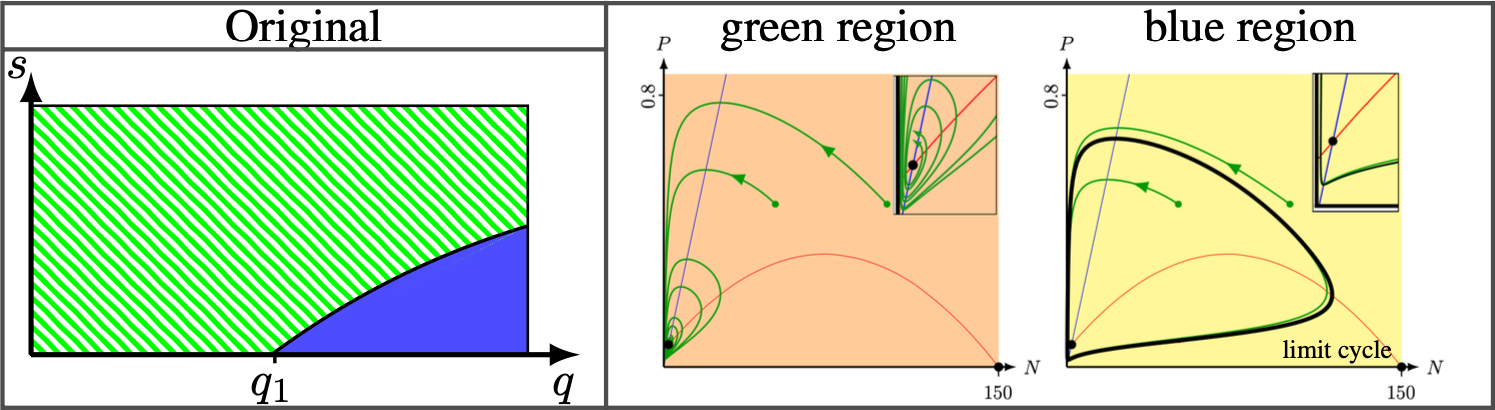}
\includegraphics[width=14cm]{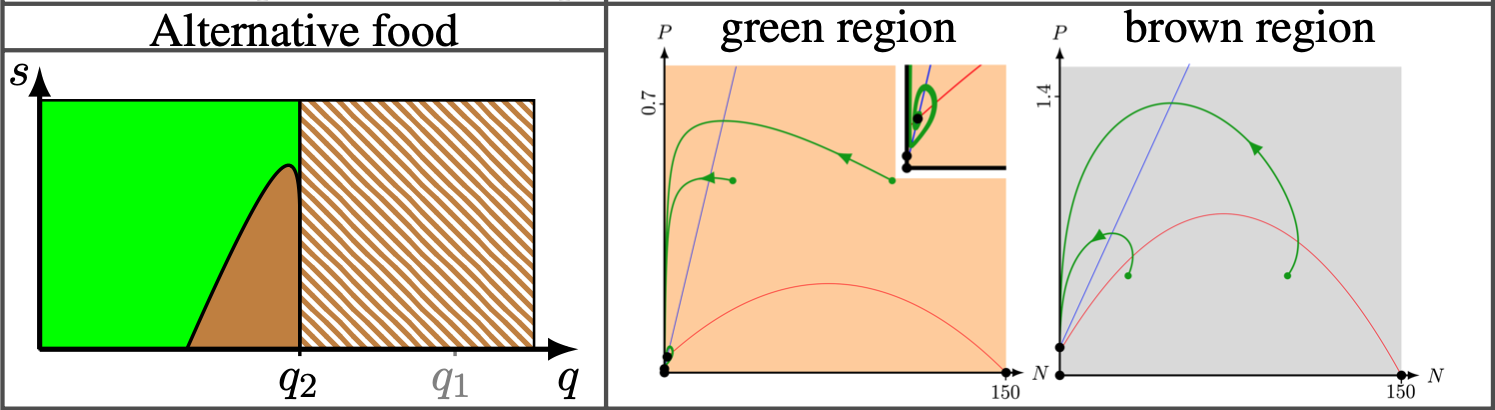}
\includegraphics[width=14cm]{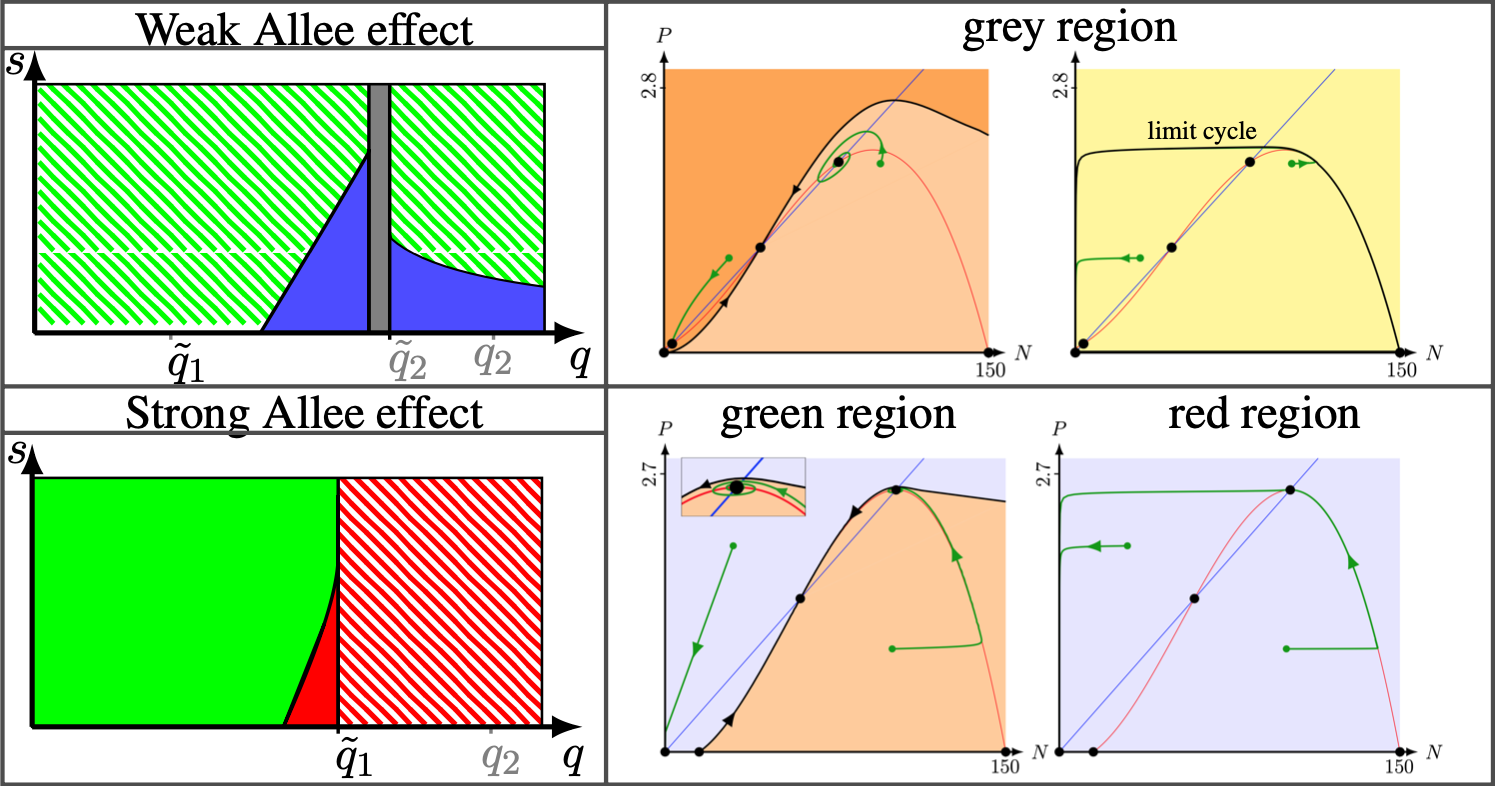}
\includegraphics[width=14cm]{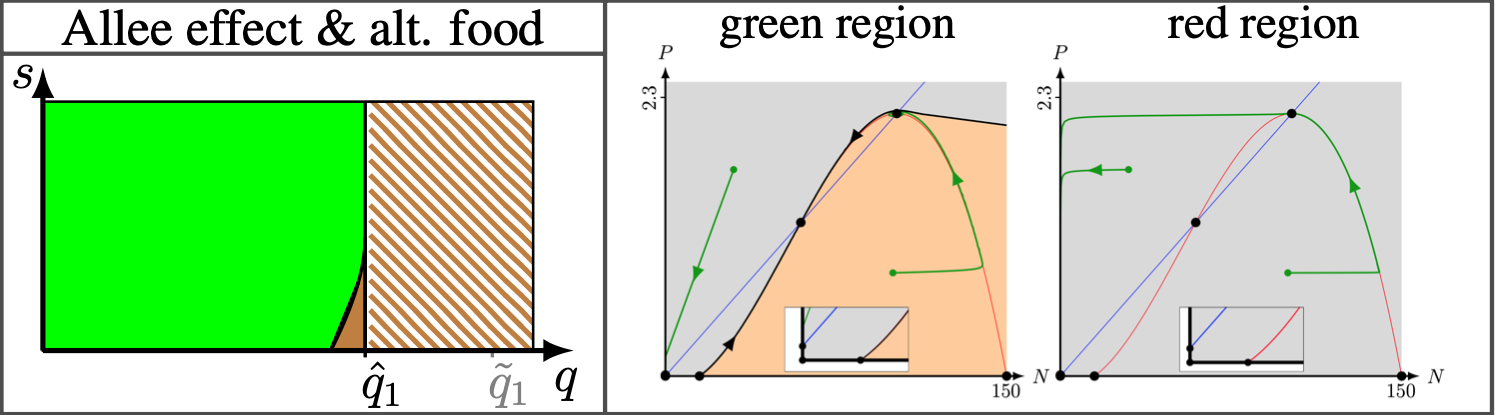}
\caption{The phase plane and bifurcation diagrams of modified May--Holling--Tanner models (~\eqref{eq5},~\eqref{eq7},~\eqref{eq8} and~\eqref{eq1}) for the system parameters $(r,K,a,n,c,m)=(4,150,6,0.025,0.01,\pm15)$ (see Table ~\ref{T02}) fixed and created with the numerical bifurcation package MATCONT~\cite{matcont2}.}  
\label{FIG05}
\end{figure}

Figure~\ref{FIG07} shows the time series of the predator and prey populations in the original May--Holling--Tanner model~\eqref{eq8}, the May--Holling--Tanner model with alternative food for the predator~\eqref{eq7} and the May--Holling--Tanner model with Allee effect on the prey~\eqref{eq5}. The system parameters are as in Table~\ref{T02} and $(q,s)=(700,1.25)$. Note that these values lie inside the solid or hatched green regions in Figure~\ref{FIG05}. Moreover, different predator and prey initial densities are taken into account to highlight the different types of behaviours. If the parameters $(q,s)$ are located in the blue region of Figure~\ref{FIG05}, then the original model supports oscillation of the solutions and it does not depend on the initial conditions, see the first column of Figure~\ref{FIG07}. That is, every initial density evolves to the coexistence of the species. If the parameters $(q,s)$ are located in the green region of Figure~\ref{FIG05}, then the May--Holling--Tanner model with strong Allee effect ($m>0$) supports both coexistence and extinction of both species. That is, there exists an initial density which evolves to the coexistence of the species, while there are also initial conditions which evolve to the extinction of both species (see second column of Figure~\ref{FIG07}). 

If the parameters $(q,s)$ are located in the grey region of Figure~\ref{FIG05}, then the May--Holling--Tanner model with weak Allee effect ($m<0$) supports the stabilisation of both populations to two different equilibrium points. That is, there exists an initial density which evolves to a coexistence point, while there are also initial conditions which evolve to a lower coexistence point (see third column of Figure~\ref{FIG01}). Finally, the modified May--Holling--Tanner model with an alternative food source for the predator supports, depending on the initial conditions, both oscillation and the extinction of the prey species only. That is, there exists as initial density which evolves to a limit cycle which represents the oscillation of the species, while there are also initial conditions which evolve to the extinction of the prey and the stabilisation of the predator population (see fourth column of Figure~\ref{FIG01}).

\begin{figure}[h]
\centering
\includegraphics[width=16cm]{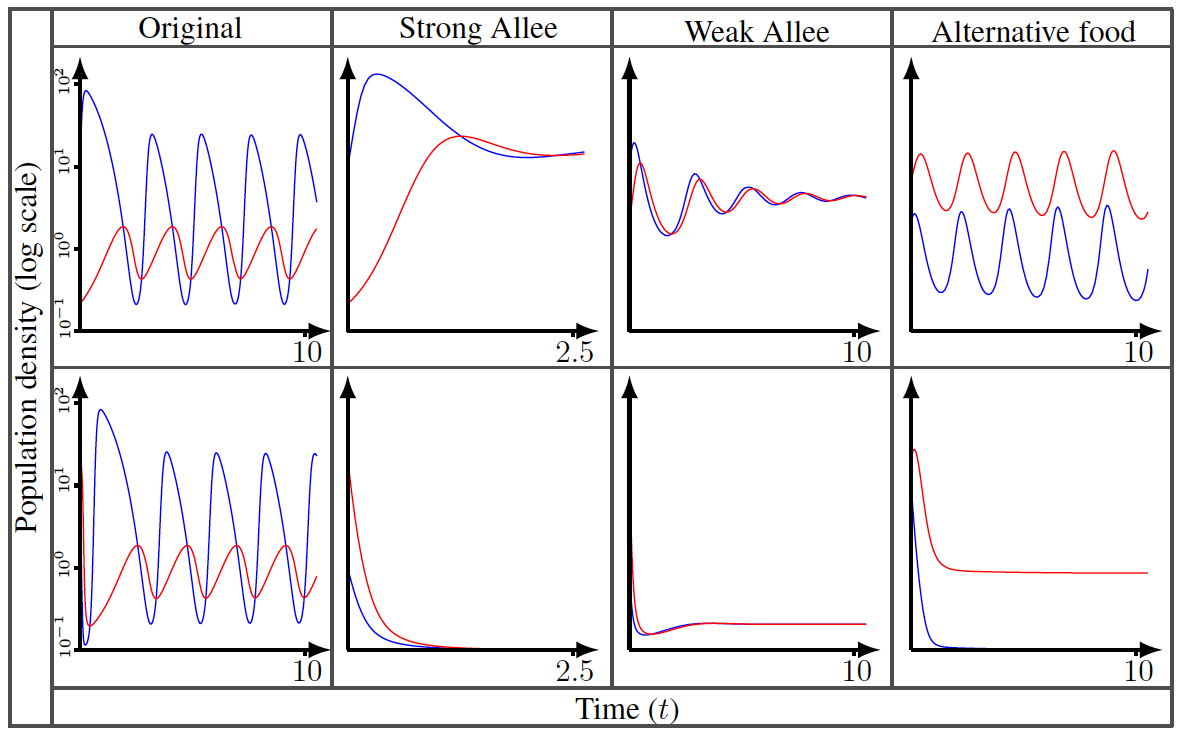}
\caption{The characteristics of predator-prey populations in a log scale in the original May--Holling--Tanner model~\eqref{eq1}, the May--Holling--Tanner model with alternative food~\eqref{eq7}, the May--Holling--Tanner model with strong Allee effect~\eqref{eq5} ($m>0$) and the May--Holling--Tanner model with weak Allee effect~\eqref{eq5} ($m<0$) over time ($t$). In all the cases, the dynamics are obtained through numerical integration considering different initial densities and the system parameters are as in Table~\eqref{T02} and $(q,s)=(700,1.25)$.}
\label{FIG07}
\end{figure}

\section{Conclusion}\label{co}
In this manuscript, the May--Holling--Tanner predator-prey model with functional response Holling type II~\eqref{eq1} was reviewed. We also considered that the prey population is affected by strong and weak Allee effect~\eqref{eq4} and the predator population has an alternative food source~\eqref{eq6}. Using a diffeomorphism we extend the analysis of a topologically equivalent system~\eqref{eq5},~\eqref{eq7},~\eqref{eq8} and~\eqref{eq1}.   

We have also extended the analysis by using the system parameters presented in~\cite{hanski,turchin}, see Table~\ref{T02}. In addition, we have shown that the predation rate per capita (parameter $q$) impacts the number of equilibrium points in the first quadrant while the intrinsic growth rate of the predator (parameter $s$) changes the behaviour of the positive(s) equilibrium point(s), see Figure~\ref{FIG05}. Therefore, it is clear that the self-regulation depends on the values of these two parameters. Consequently, we have shown that the original May--Holling--Tanner model with a strong Allee effect ($m>0$), a weak Allee effect ($m<0$) and with both modifications, i.e. Allee effects (strong and weak) for the prey and alternative food for the predator, can mean extinction and coexistence of prey and predator populations, depending on particular conditions. This is because the stable manifold of a saddle equilibrium point or an unstable limit cycle act as a separatrix between the basin of attraction of the origin and a stable equilibrium point in the first quadrant. However, the original May--Holling--Tanner model with a weak Allee effect ($m<0$) and with alternative food for the predator are present with the coexistence or oscillation of the population. This is due to the existence of at least one positive equilibrium point which can be either stable, unstable surrounded by a stable limit cycle or stable surrounded by two limit cycles. 

In Figure~\ref{FIG08} we observe that the introduction of a density dependent phenomenon in the prey population and/or an alternative food source for the predator population have an impact on the predation rate (parameter $q$). The predation rate at which the Hopf curve of the Holling--Tanner model with Allee effect~\eqref{eq5} is reduced, see the top right panel of Figure~\ref{FIG08}. In particular, the Hopf curve of the model with strong Allee effect shows that system~\eqref{eq5} can have up to two positive equilibrium points which can collapse at $q=\tilde{q}_1$ (respectively for $m<0$ at $q=\tilde{q}_2$). Therefore, for $q<\tilde{q}_1$ system~\eqref{eq5} presents two positive equilibrium points and thus the coexistence of both populations is observed~\cite{arancibia3}. 
 
\begin{figure}[h!]
\centering
\includegraphics[width=15cm]{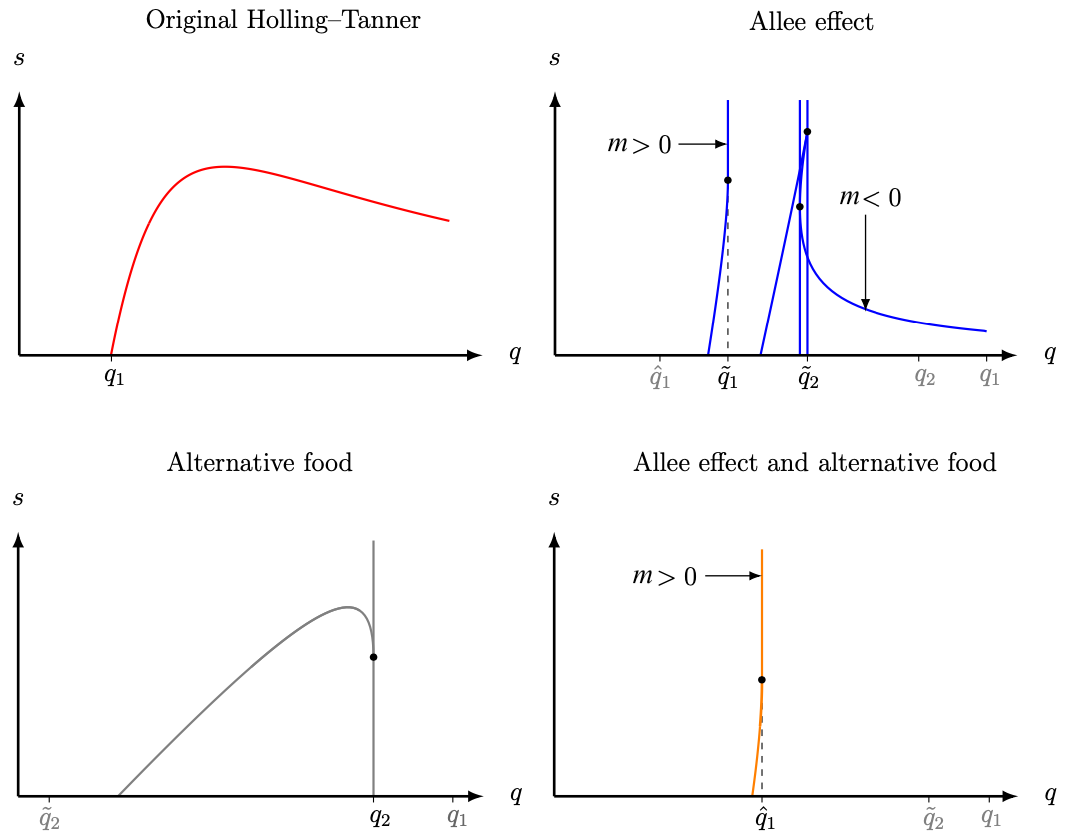}
\caption{In the top left panel we present the Hopf bifurcation curve of the original May--Holling--Tanner model~\eqref{eq1} (red curve). In the top right panel, we present the Hopf bifurcation curve of the May--Holling--Tanner model with Allee effect~\eqref{eq5} (blue curve). While, in the bottom left panel we present the Hopf bifurcation curve of the May--Holling--Tanner model with alternative food for the predator~\eqref{eq7} (grey curve) and in the bottom right panel the Hopf bifurcation curve of the May--Holling--Tanner model with both modification~\eqref{eq8} (orange curve) is presented. In all panels we used the system parameters $(r,K,a,n,c,m)=(4,150,6,0.025,0.01,\pm15)$ (see Table ~\ref{T02}) fixed and created with the numerical bifurcation package MATCONT~\cite{matcont2}.}
\label{FIG08}
\end{figure}

In contrast, the Hopf curve of the model with week Allee effect ($m<0$) shows that if $q>\tilde{q}_2$ then system~\eqref{eq5} with $m<0$ has always one positive equilibrium point which is stable or unstable surrounded by a stable limit cycle and thus the coexistence or oscillation of both population is observed. Note that the dynamic of system~\eqref{eq5} with $m<0$ is similar with the dynamics of system~\eqref{eq1} since in both systems there is always one positive equilibrium point. Consequently, there is a coexistence and/or oscillation of predator and prey population. 

We also consider the Hopf curve of the Holling--Tanner model with alternative food for predators~\eqref{eq7} and the system parameters $(r,K,a,n,c,m)=(4,150,6,0.025,0.01)$ (see Table ~\ref{T02}) fixed, see the bottom left panel of Figure \ref{FIG08}. In particular, the Hopf curve for the original May--Holling--Tanner model, considering an alternative food for predators, shows that the positive equilibrium points collapse at $q=q_2$. Therefore, for $q<q_2$ system~\eqref{eq7} present also two positive equilibrium points. One of them is always a saddle point and the other can be stable or unstable surrounded by stable limit cycle and thus the coexistence or oscillation of both population is also observed~\cite{arancibia5,arancibia9}. Note that if $q>q_2$ the equilibrium point related to the alternative food is globally stable. Therefore, system~\eqref{eq7} supports the stabilisation of only the predator population and the extinction of the prey population.

Next, we consider the Hopf curve of the May--Holling--Tanner model with both modification~\eqref{eq8}, i.e. a strong Allee effect on the prey and alternative food for the predator, and the system parameters $(r,K,a,n,c,m)=(4,150,6,0.025,0.01,\pm15)$ (see Table ~\ref{T02}) fixed, see the bottom right panel of Figure~\ref{FIG08}. In particular, the modified system~\eqref{eq8} can have up to two positive equilibrium points which can collapse at  $q=\hat{q}_1$. Therefore, for $q<\hat{q}_1$ system~\eqref{eq8} support the coexistence of both population or the extinction of only the prey population. 

In summary, the bifurcation diagrams of the models studied in~\cite{saez,arancibia3,arancibia} are often qualitatively similar but their solutions behave quantitatively differently. In other words, it is observed that there are two groups of modes which support equivalent ecological behaviour due to the addition of the modifications in the May--Holling--Tanner model. That is, an alternative food source for the predator, a strong Allee effect or the combination of both modifications support the coexistence and the extinction of at least one of the species. In contrast, the original model and the model with weak Allee effect ($m<0$) do not support the extinction of one of the species.

\section*{Acknowledgements}
This project was fully supported by Universidad de las Américas (UDLA) Chile. Especially, authors would like to give special thanks to Jaime Garcia for his support in the edition of this manuscript. 

\section*{Conflict of interest}
The authors declare that no conflict of interest.

\appendix
\section*{Appendices}

\section{The modified May--Holling--Tanner model with weak Allee effect~\eqref{eq12} ($M<0$).}

The May--Holling--Tanner model with a strong Allee effect ($m>0$) was studied in~\cite{arancibia3} in which the authors simplified the system by introducing a change of variable and time rescaling. This is given by the function $\phi :\breve{\Omega}\times\mathbb{R}\rightarrow \Omega\times\mathbb{R}$, where $\phi(u,v,\tau)=(N,P,t)=(Ku,nKv,\tau u(u+a/K)/rK)$. Moreover, by setting  
\begin{equation}\label{eq11} 
M:=\dfrac{m}{K}<1,~Q:=\dfrac{qn}{rK},~S:=\dfrac{s}{rK}~\mbox{and}~A=\dfrac{a}{K} 
\end{equation}
into system~\eqref{eq5} we obtain 
\begin{equation}\label{eq12}
\begin{aligned}
	\dfrac{du}{d\tau} & = u^2((u+A)(1-u)(u-M)-Qv) \,,\\
	\dfrac{dv}{d\tau} & =  Sv(u+A)(u-v) \,.
\end{aligned}
\end{equation}
\begin{figure}
\centering
\includegraphics[width=5.5cm]{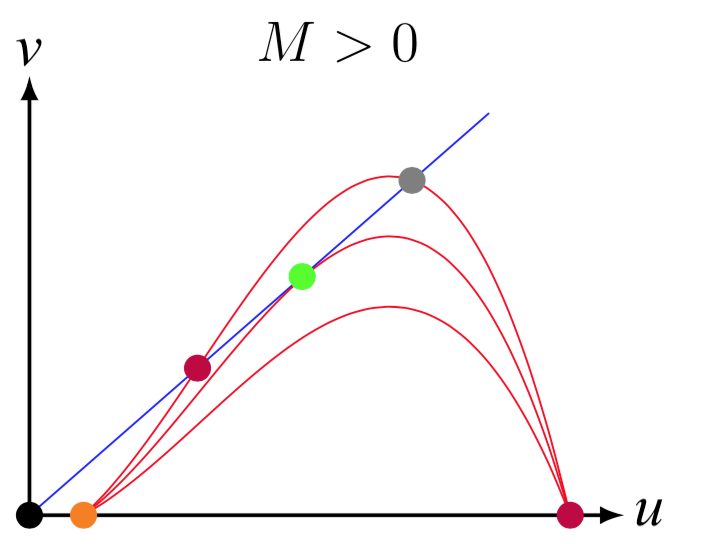}
\includegraphics[width=5.5cm]{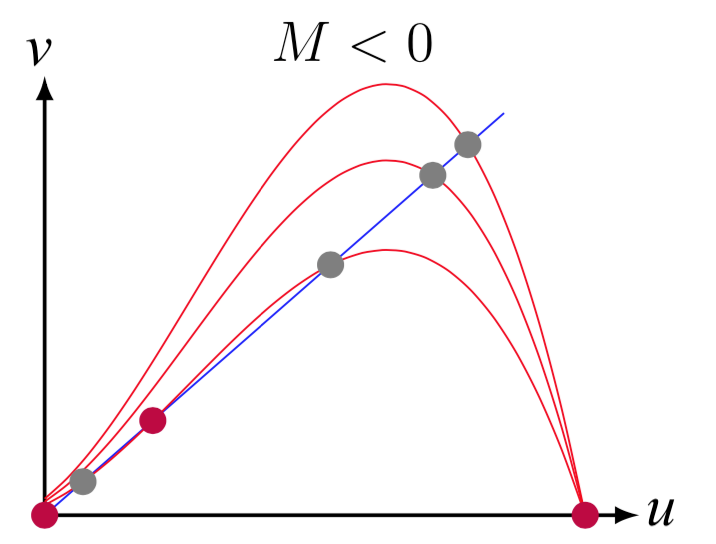}
\caption{The intersection of the predator nullcline (blue line) and the prey nullcline (red curve) in the May--Holling--Tanner model with strong ($M>0$) and weak ($M<0$) Allee effects by changing the parameter $Q$. The dark red circle represents a saddle equilibrium point, the grey circles represent an equilibrium point which can be stable or unstable, the black circle represents an equilibrium point which is always stable, the orange circle represents an equilibrium point which is always unstable and the green circle represents an equilibrium point which is the collision of two equilibrium points.}
\label{FIG03}
\end{figure} 

To analyse the positive solution(s) of system~\eqref{eq12} we consider three cases of the intersection of the nullcline of system~\eqref{eq12} and $M<0$. The $u$-nullclines of system~\eqref{eq12} are $u=0$ and $v=(u+A)(1-u)(u-M)/Q$, while the $v$-nullclines are $v=0$ and $v=u$. Hence, the equilibrium points for system~\eqref{eq12} are $(0,0)$, $(1,0)$ and the point(s) $(u^*,v^*)$ with $v^*=u^*$ and where $u^*$ is determined by the solution(s) of
\begin{equation}\label{eq16}
\begin{aligned}
u^3-(M+1-A)u^2-(A(M+1)-Q-M)u+AM=u^3-H(A,M)u^2-L(A,M,Q)u+AM=0,
\end{aligned}
\end{equation}
where $H(A,M)=M+1-A$ and $L(A,M,Q)=A(M+1)-Q-M$. System~\eqref{eq12} always has one positive equilibrium point, which we denote by $P_1=(u_1,u_1)$ where $0<u_1<1$. Factoring out $(u-u_1)$ from the equation resultant~\eqref{eq16} of the intersection of the nullcline of system~\eqref{eq12}~\cite{arancibia3}, leaves as with second order polynomial equation
\begin{equation}\label{eq17}
u^2-(1-A+M-u_1)u-\dfrac{AM}{u_1}=0
\end{equation}
Therefore, if $1-A+M-u_1>0$ then the roots of~\eqref{eq17} are given by
\begin{equation}\label{eq18}
u_{2,3}=\dfrac{1}{2}\left(1-A+M-u_1\pm\sqrt{\Delta}\right),~\text{with}~\Delta=(1-A+M-u_1)^2+\dfrac{4AM}{u_1},
\end{equation}
such that $0<u_2<u_3<1$.

\subsection{Stability of positive equilibrium points:}

To determinate the stability of the positive equilibrium points $P_1$, $P_2$, and $P_3$ of system~\eqref{eq12} with $M<0$ we compute the Jacobian matrix~\eqref{eq23_M} at these equilibrium points
\begin{equation}\label{eq23_M}
J(u_i,u_i)=\begin{pmatrix}
u_i^2((1-u_i)(u_i-M)+(u_i+A)(1-u_i)-(u_i+A)(u_i-M)) & -Qu_i^2 \\ 
Su_i(u_i+A)  &  -Su_i(u_i+A) 
\end{pmatrix},
\end{equation}
with $i=1,2,3$. Thus, the trace and the determinant of the Jacobian matrix~\eqref{eq23_M} are
\begin{equation}\label{eq19}
\begin{aligned}
\det{(P_{i})}&=Su_i^2(u_i+A)h(u_i)\\
\tr{(P_{i})}&=u_i(u_i+A)\left(f(u_i)-S\right),	
\end{aligned}
\end{equation}
with 
\begin{equation}\label{eq19_2}
\begin{aligned}
h(u_i)&=u_i^2(2u_i-(1-A+M))-AM~\text{and}\\
f(u_i)&=\dfrac{u_i}{u_i+A}\left((1-u_i)(u_i-M)+(u_i+A)(1-u_i)-(u_i+A)(u_i-M)\right).
\end{aligned}
\end{equation}
Then, we consider the stability of these positive equilibrium points $P_1$, $P_2$ and $P_3$ of system~\eqref{eq12} with $M<0$ in the interior of $\Phi=\{\left(u,v\right),\ 0<u\leq1,\ v\geq0\}$.

\begin{lemm}\label{CH8:N1}
Let the system parameters of~\eqref{eq12} be such that $1-A+M\leq0$ or $1-A+M>0$ and $A(M+1)-Q-M\geq0$. Then system~\eqref{eq12} has only one positive equilibrium point $(u_1,u_1)$ which can be
\begin{enumerate}[label=(\roman*)]
\item a repeller if $S<f(u_1)$, 
\item an attractor if $S>f(u_1)$,
\end{enumerate}
with $f(u_*)$ defined in \eqref{eq19_2}.
\end{lemm}
\begin{proof}
Evaluating $u^2(2u-(1-A+M))-AM$ at $u=u_1$ gives: 
$$u^2(2u-(1-A+M))-AM=u_1^2(2u_1-(1-A+M))-AM.$$ 
Then,
\begin{enumerate}[label=(\roman*)]
\item if $1-A+M<0$ then $u_1^2(2u_1-(1-A+M))-AM>0$. Hence $\det(J(P_1))>0$ since $M<0$,
\item if $1-A+M=0$ then $2u_1^3-AM>0$. Hence $\det(J(P_1))>0$,
\item if $1-A+M>0$ then $u_1^2(2u_1-(1-A+M))-AM=u_1^3+(A(1+M)-Q-M)u_1-2AM>0$. Hence $\det(J(P_1))>0$ since $A(M+1)-Q-M\geq0$.
\end{enumerate}
Therefore, the sign of the trace~\eqref{eq19}, and thus the behaviour of $(u_1,u_1)$ depends on the parity of $f(u_1)-S$. 
\end{proof}

\begin{lemm}\label{CH8:N2}
Let the system parameters of~\eqref{eq12} be such that $1-A+M>0$, $A(M+1)-Q-M<0$ and $\Delta>0$~\eqref{eq18}. Then the equilibrium point $P_1$ is
\begin{enumerate}[label=(\roman*)]
\item a saddle point if $h(u_1)<0$
\item a repeller if $h(u_1)>0$ and $S<f(u_1)$, 
\item an attractor if $h(u_1)>0$ and $S>f(u_1)$.
\end{enumerate}
with $f(u_*)$ and $h(u_*)$ defined in \eqref{eq19_2}.
\end{lemm}
\begin{proof}
Evaluating $u^2(2u-(1-A+M))-AM$ at $u=u_1$ gives: 
$$u^2(2u-(1-A+M))-AM=u_1^2(2u_1-(1-A+M))-AM.$$ 
Then, the sign of the determinant and the trace~\eqref{eq19}, and thus the behaviour of $P_1$ depends on the parity of $h(u_1)$ and $f(u_1)-S$. 
\end{proof}

\begin{lemm}\label{CH8:p2}
Let the system parameters of~\eqref{eq12} be such that $1-A+M>0$, $A(M+1)-Q-M<0$ and $\Delta>0$~\eqref{eq18}. Then the equilibrium point $P_2$ is
\begin{enumerate}[label=(\roman*)]
\item a saddle point if $u_2>u_1$
\item a repeller if $u_2<u_1$ and $S<f(u_2)$, 
\item an attractor if $u_2<u_1$ and $S>f(u_2)$,
\end{enumerate}
with $f(u_*)$ defined in \eqref{eq19_2}.
\end{lemm}
\begin{proof}
Evaluating $u^2(2u-(1-A+M))-AM$ at $u=u_2=(1-A+M-u_1-\sqrt{\Delta})/2$ gives: 
\[\begin{aligned}
u^2(2u-(1-A+M))-AM=&\dfrac{\sqrt{\Delta}}{4}(-1-A-M+3u_1+\sqrt{\Delta})\\
&(1-A+M-u_1-\sqrt{\Delta})=u_2\sqrt{\Delta}(u_1-u_2). 
\end{aligned}\]
Then, the sign of the determinant depends on the parity of $u_1-u_2$. Then, the trace~\eqref{eq19}, and thus the behaviour of $P_2$ depends on the parity of $f(u_2)-S$. 
\end{proof}
\begin{lemm}\label{CH8:p3}
Let the system parameters of~\eqref{eq12} be such that $1-A+M>0$, $A(M+1)-Q-M<0$ and $\Delta>0$~\eqref{eq18}. Then the equilibrium point $P_3$ is
\begin{enumerate}[label=(\roman*)]
\item a saddle point if $u_3<u_1$
\item a repeller if $u_3>u_1$ and $S<f(u_3)$, 
\item an attractor if $u_3>u_1$ and $S>f(u_3)$,
\end{enumerate}
with $f(u_*)$ defined in \eqref{eq19_2}.
\end{lemm}
\begin{proof}
Evaluating $u^2(2u-(1-A+M))-AM$ at $u=u_3=(1-A+M-u_1+\sqrt{\Delta})/2$ gives: 
\[\begin{aligned}
u^2(2u-(1-A+M))-AM=&\dfrac{\sqrt{\Delta}}{4}(-1+A-M+3u_1-\sqrt{\Delta})\\
&(-1+A-M+u_1-\sqrt{\Delta})=u_3\sqrt{\Delta}(u_3-u_1). 
\end{aligned}\]
Then, the sign of the determinant depends on the parity of $u_3-u_1$. Then, the trace~\eqref{eq19}, and thus the behaviour of $P_3$ depends on the parity of $f(u_3)-S$. 
\end{proof}


\section{The modified May--Holling--Tanner model with alternative food for the predator and Allee effect.}

The May--Holling--Tanner model with Allee effect on the prey and alternative food for predators is given by system~\eqref{eq8}. The system can be simplified by introducing a change of variable and time rescaling presented in~\cite{arancibia}. It is given by the function $\zeta :\breve{\Psi}\times\mathbb{R}\rightarrow \Psi\times\mathbb{R}$, where $\zeta(u,v,\tau)=(N,P,t)=(Ku,nKv,\tau(u+a/K)(u+c/(nK))/rK)$~\cite{arancibia}. By setting the parameters $A$, $C$, $Q$, $S$ and $M$ defined in~\eqref{eq11}, into system~\eqref{eq8} we obtain  
\begin{equation}\label{eq15}
\begin{aligned}
	\dfrac{du}{d\tau} & = u(u+C)((u+A)(1-u)(u-M)-Qv) \,,\\
	\dfrac{dv}{d\tau} & =  Sv(u+A)(u-v+C) \,,
\end{aligned}
\end{equation}
\begin{figure}
\centering
\includegraphics[width=5cm]{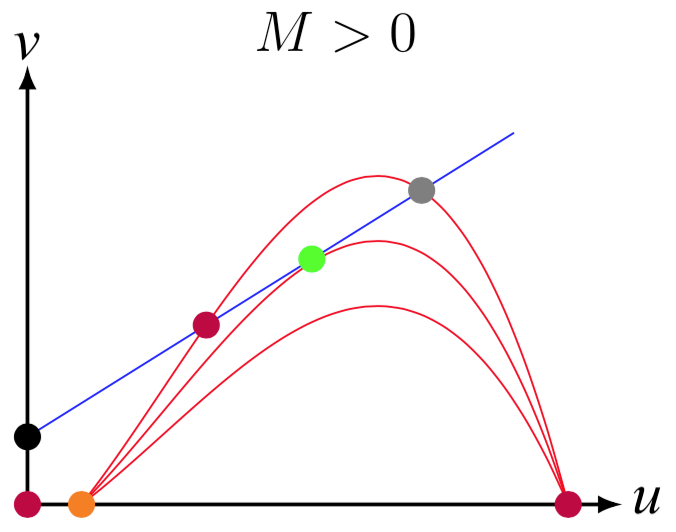}
\includegraphics[width=5cm]{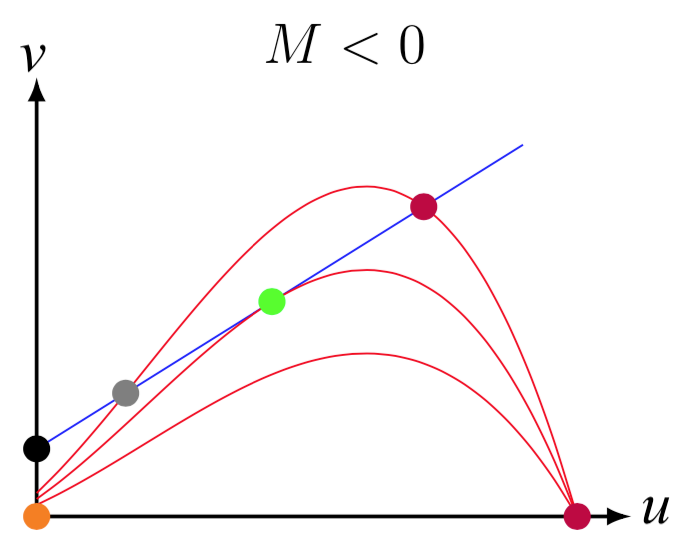}
\caption{The intersection of the predator nullcline (blue line) and the prey nullcline (red curve) in the May--Holling--Tanner model with an alternative food for predators and strong ($M>0$) and weak ($M<0$) Allee effect by changing the parameter $Q$. The purple circle represents a saddle equilibrium point, the grey circle represents an equilibrium point which can be stable or unstable, the black circle represents an equilibrium point which is always stable and the orange circle represents an equilibrium point which is always unstable.}
\label{FIG06}
\end{figure} 

To analyse the positive solution(s) of system~\eqref{eq15} we consider three cases of the intersection of the nullcline of system~\eqref{eq15}. The $u$-nullclines of system~\eqref{eq15} are $u=0$ and $v=(u+A)(1-u)(u-M)/Q$, while the $v$-nullclines are $v=0$ and $v=u+C$. Hence, the equilibrium points for system~\eqref{eq15} are  $(0,0)$, $(1,0)$, $(M,0)$ and the point(s) $(u^*,v^*)$ with $v^*=u^*+C$ and where $u^*$ is determined by the solution(s) of
\begin{equation}\label{eq20}
u^3-(M+1-A)u^2-(A(M+1)-Q-M)u+AM+CQ=0.
\end{equation}
Define the functions $g(u)=(u+A)(1-u)(u-M)$ and $h(u)=Q(u+C)$ and observe that $\lim\limits_{u \rightarrow \pm \infty} g(u)= \mp \infty$ and $g(0)=-AM<0$. So,~\eqref{eq20} will always have a single negative real root, which we denote by $u=-G$ where $G>0$. Factoring out $(u+G)$ from~\eqref{eq20} leaves us with second order polynomial
\begin{equation}\label{eq21}
u^2-(G+M+1-A)u+(M+Q-A(M+1)+G(G-A+M+1))=0,
\end{equation}
and thus we get $G(M+Q-A(M+1)+G(G-A+M+1))-AM-CQ=0$. Note that since $Q>0$, then we obtain that $A<G<C$ or else, $A>G>C$. In addition, we get that $Q=(G+1)(G+M)(G-A)/(C-G)$. 
The solution of~\eqref{eq21} are:
\begin{equation}\label{eq22}
\begin{aligned} 
u_{1,2} &= \dfrac{1}{2}\left(1-A+G+M \pm \sqrt{\Delta} \right)\,, \quad \text{with} \\
\Delta&=(G-A+M+1)^2-4(M+Q-A(M+1)+G(G-A+M+1))\,,
\end{aligned} 
\end{equation}
such that $M<u_1\leq E \leq u_2<1$, where $E=(1-A+M+G)/2$.
In particular, 
\begin{enumerate}
\item if $\Delta<0$~\eqref{eq22}, then~\eqref{eq15} has no equilibrium points in the first quadrant;
\item if $\Delta>0$~\eqref{eq22}, then~\eqref{eq15} has two equilibrium points $P_{1,2}=(u_{1,2},u_{1,2}+C)$ in the first quadrant; and
\item if $\Delta=0$~\eqref{eq22}, then~\eqref{eq15} has one equilibrium point $(E,E+C)$ of order two in the first quadrant,
\end{enumerate} 
Additionally, in Figure \ref{FIG06} we observe that by changing the parameter $Q$ and if $M<0$ (weak Allee effect), then system~\eqref{eq15} also has no equilibrium points, two equilibrium points or one equilibrium point of order two in the first quadrant.

\subsection{Stability of positive equilibrium points}
To determine the nature of the equilibrium points we compute the Jacobian matrix $J(u,v)$ of~\eqref{eq15}
\begin{equation}\label{eq23}
J(u,v)=\begin{pmatrix}
u(u+C)g'(u)+(2u+C)(g(u)-Qv)  & -Qu(u+C) \\ 
Sv(A+C+2u-v)  &  S(C+u-2v)(u+A) 
\end{pmatrix},
\end{equation}
with $g(u)=(u+A)(1-u)(u-M)$ and $g'(u)=(1-u)(u-M)+(u+A)(1-u)-(u+A)(u-M)$.
\begin{lemm}
The equilibrium point $(0,0)$ is a saddle point, $(1,0)$ is a saddle point, $(M,0)$ is an unstable point and $(0,C)$ is a stable point.
\end{lemm}
\begin{proof}
The Jacobian matrix~\eqref{eq23} evaluated at $(0,0)$ and $(1,0)$ gives
	\[ J(0,0)=\begin{pmatrix}
	-ACM  & 0 \\ 
	0  &  ACS 
	\end{pmatrix}~\text{and}\]
	\[ J(1,0)=\begin{pmatrix}
	-(1-M)(C+1)(A+1)  & -Q(C+1) \\ 
	0  &  S(C+1)(A+1) 
	\end{pmatrix}.\]
Similarly, the Jacobian matrix~\eqref{eq23} evaluated at $(M,0)$ gives
	\[J(M,0)=\begin{pmatrix}
	M(1-M)(C+M)(A+M)  & -MQ(C+M) \\ 
	0  & S(A+M)(C+M) 
	\end{pmatrix}.\]
Finally, the Jacobian matrix~\eqref{eq23} evaluated at $(0,C)$ gives
	\[J(0,C)=\begin{pmatrix}
	-C(AM+QC)  & 0 \\ 
	ACS  &  -ACS 
	\end{pmatrix}.\]
Since $M<1$, it follows that $(0,0)$ and $(1,0)$ are a saddle point, $(M,0)$ is an unstable point and $(0,C)$ is a stable point in system~\eqref{eq15}.
\end{proof}

Next, we consider the stability of the positive equilibrium points $P_{1}$ and $P_{2}$ of system~\eqref{eq15}. The Jacobian matrix~\eqref{eq23} of system~\eqref{eq15} at these equilibrium points becomes
\begin{equation}\label{eq24}
J(u,u+C)=\begin{pmatrix}
u(u+C)g'(u)  & -Qu(u+C) \\ 
S(u+C)(u+A)  &  -S(u+C)(u+A) 
\end{pmatrix},
\end{equation}
Here, the determinant and the trace of the Jacobian matrix~\eqref{eq24} are given by
\begin{equation}\label{eq25}
\begin{aligned}
\det(J(u,u+C))&=Su(u+A)u+C)^2(Q-g'(u))\\
\tr(J(u,u+C))&= (u+C)(u+A)\left(\dfrac{ug'(u)}{u+A}-S\right)
\end{aligned}
\end{equation}
Therefore, the sign of the determinant the trace~\eqref{eq25}, and thus the stability of the positive equilibrium points, depends on the sign of $Q-g'(u_{1,2})$ and $ug'(u)/(u+A)-S$ respectively. This gives the following results.
\begin{lemm}
Let the system parameters of~\eqref{eq15} be such that $\Delta>0$~\eqref{eq25}. Then, the equilibrium point $P_1$ is a saddle point.
\end{lemm}
\begin{proof}
Evaluating $Q-g'(u)$ at $u_1$ gives:
$$Q-g'(u_1)=-\dfrac{\sqrt{\Delta}}{2}\left(3G+M+1-A-\sqrt{\Delta}\right)$$
Hence, $\det(J(P_1))<0$ and $P_1$ is thus a saddle point.
\end{proof}
\begin{lemm}
Let the system parameters be such that $\Delta>0$~\eqref{eq25}. Then, the equilibrium point $P_2$ is:
\begin{enumerate}
\item unstable if $S<\dfrac{D}{4\left(1+A+M+G+\sqrt{\Delta}\right)}$; and
\item stable if $S>\dfrac{D}{4\left(1+A+M+G+\sqrt{\Delta}\right)}$,
\end{enumerate}
with $D=\left(1-A+M+G+\sqrt{\Delta}\right)\left(1+A-M-G-\sqrt{\Delta}\right)\left(1-A-M+G+\sqrt{\Delta}\right)+2\left(1+A+M+G+\sqrt{\Delta}\right)\left(A-G-\sqrt{\Delta}\right)$. 
\end{lemm}
\begin{proof}
Evaluating $Q-g'(u)$ at $u_2$ gives:
$$Q-g'(u_1)=-\dfrac{\sqrt{\Delta}}{2}\left(3G+M+1-A+\sqrt{\Delta}\right)$$
Hence, $\det(J(P_2))>0$. Evaluating $ug'(u)/(u+A)-S$ at $u_2$ gives
\[\begin{aligned}
\dfrac{u_2g'(u_2)}{u_2+A}-S=&\dfrac{(1-A+M+G+\sqrt{\Delta})(1+A-M-G-\sqrt{\Delta})(1-A-M+G+\sqrt{\Delta})}{4(1+A+M+G+\sqrt{\Delta})}\\
&+\dfrac{(A-G-\sqrt{\Delta})}{2}-S\\
\end{aligned}\]
Therefore, the sign of the trace, and thus the behaviour of $P_2$ depends on the sign of $u_2g'(u_2)/(u_2+A)-S$. 
\end{proof}  

Next we analyse the case when $\Delta=0$~\eqref{eq25} then the equilibrium points $P_1$ and $P_2$ collapse such that  $u_1=u_2=E=(1-A+M+G)/2$. Therefore, system~\eqref{eq15} has one equilibrium point of order two in the first quadrant given by $(E,E+C)$.
\begin{lemm}
Let the system parameters be such that $\Delta=0$~\eqref{eq22}. Then, the equilibrium point $(E,E+C)$ is:
\begin{enumerate}
\item a stable saddle-node if $S<\dfrac{H}{4(1+A+M+G)}$; and
\item an unstable saddle-node if $S>\dfrac{H}{4(1+A+M+G)}$,
\end{enumerate}
with $H=(1-A+M+G)(1+A-M-G)(1-A+G-M)+2(1+M+G+A)(A-G)$.
\end{lemm}
\begin{proof}
Evaluating $Q-g'(u)$ at $u=E$ gives    
$$Q-g'(E)=0.$$
Therefore, $\det(J(E,E))=0$. Next, evaluating $ug'(u)/(u+A)-S$ at $E$ gives
\[\begin{aligned}
Eg'(E)/(E+A)-S =&\dfrac{(1-A+M+G)(1+A-M-G)(1-A-M+G)}{4(1+A+M+G)}\\
&+\dfrac{(A-G)}{2}-S.
\end{aligned}\]
Therefore, the behaviour of the equilibrium point $(E,E+C)$ depends on the parity of $Eg'(E)/(E+A)-S$. 
\end{proof}



\begin{thebibliography}{99}
\bib{aguirre1}{article}{
 	author={Aguirre, P.},
	author={Gonz{\'a}lez-Olivares, E.},
	author={S{\'a}ez, E.},
 	title={Two limit cycles in a {Leslie--Gower} predator--prey model with additive {Allee} effect},
	date={2009},
	journal={Nonlinear Analysis: Real World Applications},
	volume={10},
	pages={1401\ndash 1416},}

\bib{allee2}{book}{
	author={Allee, W.},
	title={The social life of animals},
	publisher={WW Norton \& Co},
	address={New York},
	date={1938},}

\bib{allee}{book}{
	author={Allee, W.},
	author={Park, O.},
	author={Emerson, A.},
	author={Park, T.},
	author={Schmidt, K.},
	title={Principles of animal ecology},
	publisher={WB Saundere Co. Ltd.},
	address={Philadelphia},
	date={1949},}

\bib{andersson}{article}{
	author={Andersson, M.},
	author={Erlinge, S.},
	title={Influence of predation on rodent populations},
	date={1977},
	journal={Oikos},
	volume={29},
	pages={591\ndash 597},}

\bib{arancibia7}{article}{
	author={Arancibia-Ibarra, C.},
	title={The basins of attraction in a {Modified May-Holling-Tanner} predator-prey model with {Allee} effect},
	date={2019},
	journal={Nonlinear Analysis: Theory, Methods \& Applications},
	volume={185},
	pages={15\ndash 28},}

\bib{arancibia3}{article}{
	author={Arancibia-Ibarra, C.},
	author={Flores, J.},
	author={Pettet, G.},
	author={van Heijster, P.},
	title={A {Holling--Tanner} predator--prey model with strong {Allee}effect},
	date={2019},
	journal={{International Journal} of {Bifurcation} and {Chaos}},
	volume={29},
	pages={1\ndash 16},}

\bib{arancibia}{article}{
	author={Arancibia-Ibarra, C.},
	author={Gonz{\'a}lez-Olivares, E.},
	title={A modified {Leslie--Gower} predator--prey model with hyperbolic functional response and {Allee} effect on prey},
	date={2011},
	journal={BIOMAT 2010 International Symposium on Mathematical and Computational Biology},
 	pages={146\ndash 162},}

\bib{aziz}{article}{
 	author={Aziz-Alaoui, M.},
	author={Daher, M.},
	title={Boundedness and global stability for a predator--prey model with modified {Leslie--Gower} and {Holling}--type {II} schemes},
	date={2003},
	 journal={Applied Mathematics Letters},
	volume={16},
	pages={1069\ndash 1075},}

\bib{arancibia2}{article}{
	author={Arancibia-Ibarra, C.},
	author={Gonz{\'a}lez-Olivares, E.},
	title={The {Holling--Tanner} model considering an alternative food for predator},
	date={2015},
	journal={Proceedings of the 2015 International Conference on Computational and Mathematical Methods in Science and Engineering CMMSE 2015},
 	pages={130\ndash 141},}

\bib{berec}{article}{
	author={Berec, L.},
	author={Angulo, E.},
	author={Courchamp, F.},
	title={Multiple {Allee} effects and population management},
	date={2007},
	journal={Trends in Ecology \& Evolution},
	volume={22},
	pages={185\ndash 191},}

\bib{buffoni}{article}{
	author={Buffoni, G.},
	author={Groppi, M.},
	author={Soresina, C.},
	title={Dynamics of predator--prey models with a strong {Allee} effect on the prey and predator-dependent trophic functions},
 	date={2016},
	journal={Nonlinear Analysis: Real World Applications},
	volume={30},
 	pages={143\ndash 169},}

\bib{cao}{article}{
	author={Cao, H.},
	author={Yue, Z.},
	author={Zhou, Y.},
	title={The stability and bifurcation analysis of a discrete Holling-Tanner model},
	date={2013},
	journal={Advances in Difference Equations},
	volume={2013},
	pages={1\ndash 17},}

\bib{courchamp2}{book}{
	author={Courchamp, F.},
	author={Berec, L.},
	author={Gascoigne, J.},
	title={Allee effects in ecology and conservation},
	publisher={Oxford University Press},
	date={2008},}

\bib{dalkvist}{article}{
	author={Dalkvist, T.},
	author={Sibly, R.},
	author={Topping, C.},
	title={How predation and landscape fragmentation affect vole population dynamics},
	date={2011},
	journal={PloS one},
	volume={6},
	pages={1\ndash 8},}

\bib{matcont2}{article}{
	author={Dhooge, A.},
	author={Govaerts, W.},
	author={Kuznetsov, Y.},
	author={Meijer, H.},
	author={Sautois, B.},
	title={New features of the software {MatCont} for bifurcation analysis of dynamical systems},
 	date={2008},
 	journal={{Mathematical and Computer Modelling of Dynamical Systems}},
	volume={14},
	pages={147\ndash 175},}

\bib{dumortier}{book}{
	author={Dumortier, F.},
	author={Llibre, J.},
	author={Art{\'e}s, J.},
 	title={Qualitative theory of planar differential systems},
	publisher={Springer Berlin Heidelberg},
	address={Springer-Verlag Berlin Heidelberg},
	date={2006},}

\bib{erlinge}{article}{
	author={Erlinge, S.},
 	title={Predation and noncyclicity in a microtine population in southern {Sweden}},
	date={1987},
	journal={Oikos},
	volume={50},
	pages={347\ndash 352},}

\bib{arancibia5}{article}{
	author={Gonz{\'a}lez-Olivares, E.},
	author={Arancibia-Ibarra, C.},
	author={Rojas-Palma, A.},
	author={Gonz{\'a}lez-Ya{\~n}ez, B.},
 	title={Bifurcations and multistability on the {May--Holling--Tanner} predation model considering alternative food for the predators},
	date={2019},
	journal={Mathematical Biosciences and Engineering},
 	volume={16},
 	pages={4274\ndash 4298},}

\bib{gonzalez3}{article}{
	author={Gonz{\'a}lez-Olivares, E.},
	author={Gallego-Berr{\'\i}o, L.},
	author={Gonz{\'a}lez-Ya{\~n}ez, B.},
	author={Rojas-Palma, A.},
 	title={Consequences of weak {Allee} effect on prey in the {May--Holling--Tanner} predator--prey model},
	date={2015},
	journal={Mathematical Methods in the Applied Sciences},
 	volume={38},
 	pages={5183\ndash 5186},}

\bib{arancibia9}{article}{
	author={Gonz{\'a}lez-Olivares, E.},
	author={Arancibia-Ibarra, C.},
	author={Rojas-Palma, A.},
	author={Gonz{\'a}lez-Ya{\~n}ez, B.},
 	title={Dynamics of a modified Leslie-Gower predation model considering a generalist predator and the hyperbolic functional response},
	date={2019},
	journal={Mathematical Biosciences and Engineering},
 	volume={16},
 	pages={7995\ndash 8024},}

\bib{graham}{article}{
	author={M, I.~Graham},
	author={Lambin, X.},
 	title={The impact of weasel predation on cyclic field-vole survival: the specialist predator hypothesis contradicted},
	date={2002},
	journal={Journal of Animal Ecology},
 	volume={71},
 	pages={946\ndash 956},}

\bib{hanski2}{article}{
	author={Hanski, I.},
	author={Hansson, L.},
	author={Henttonen, H.},
	title={Specialist predators, generalist predators, and the microtine rodent cycle},
	date={1991},
	journal={The Journal of Animal Ecology},
	volume={60},
	pages={353\ndash 367},}

\bib{hanski}{article}{
	author={Hanski, I.},
	author={Henttonen, H.},
	author={Korpim{\"a}ki, E.},
	author={Oksanen, L.},
	author={Turchin, P.},
 	title={Small--rodent dynamics and predation},
	date={2001},
	journal={Ecology},
 	volume={82},
 	pages={1505\ndash 1520},}

\bib{hanski3}{article}{
	author={Hanski, I.},
	author={Turchin, P.},
	author={Korpimaki, E.},
	author={Henttonen, H.},
 	title={Population oscillations of boreal rodents: regulation by mustelid predators leads to chaos},
	date={1993},
	journal={Nature},
 	volume={364},
 	pages={232\ndash 235},}

\bib{hansson}{article}{
	author={Hansson, L.},
 	title={Competition between rodents in successional stages of taiga forests: {Microtus} agrestis vs. {Clethrionomys} glareolus},
	date={1983},
	journal={Oikos},
	volume={40},
	pages={258\ndash 266},}

\bib{huang}{article}{
	author={Huang, J.},
	author={Gong, Y.},
	author={Ruan, S.},
 	title={Bifurcation analysis in a predator-prey model with constant-yield predator harvesting},
	date={2013},
	journal={Discrete and Continuous Dynamical Systems Series B},
 	volume={18},
 	pages={2101\ndash 2121},}

\bib{jankovic}{article}{
	author={Jankovic, M.}, 
	author={Petrovskii, S.}, 
	title={Are time delays always destabilizing? {Revisiting} the role of time delays and the {Allee} effect}, 
	date={2014}, 
	journal={Theoretical Ecology}, 
	volume={7}, 
	pages={335\ndash 349},}

\bib{leslie}{article}{
	author={Leslie, P.}, 
	author={Gower, J.}, 
	title={The Properties of a Stochastic Model for the Predator-Prey Type of Interaction Between Two Species}, 
	date={1960}, 
	journal={Biometrika}, 
	volume={47}, 
	pages={219\ndash 234},}

\bib{liermann}{article}{
	author={Liermann, M.},
	author={Hilborn, R.},
 	title={Depensation: evidence, models and implications},
	date={2001},
	journal={Fish and Fisheries},
 	volume={2},
 	pages={33\ndash 58},}

\bib{lisena}{article}{
	author={Lisena, B.},
 	title={Global stability of a periodic Holling-Tanner predator-prey model},
	date={2018},
	journal={Mathematical Methods in the Applied Sciences},
 	volume={41},
 	pages={3270\ndash 3281},}

\bib{madison}{article}{
	author={Madison, D.},
	author={FitzGerald, R.},
	author={McShea, W.},
 	title={Dynamics of social nesting in overwintering meadow voles ({Microtus} pennsylvanicus): possible consequences for population cycling},
	date={1984},
	journal={Behavioral Ecology and Sociobiology},
 	volume={15},
 	pages={9\ndash 17},}

\bib{may}{book}{
	author={May, R.},
 	title={Stability and complexity in model ecosystems},
 	series={Monographs in population biology},
	publisher={{Princeton} {University} {Press}},
	address={Princeton, N.J.},
	date={1974},
 	volume={6},}

\bib{mcdonald}{article}{
	author={McDonald, R.},
	author={Webbon, C.},
	author={Harris, S.},
 	title={The diet of stoats {(Mustela erminea)} and weasels {(Mustela nivalis)} in {Great Britain}},
	date={2000},
	journal={Journal of Zoology},
 	volume={252},
 	pages={363\ndash 371},}

\bib{morris}{article}{
	author={Morris, D.},
 	title={{Measuring} the {Allee} effect: positive density dependence in small mammals},
	date={2002},
	journal={Ecology},
 	volume={83},
 	pages={14\ndash 20},}

\bib{pielou}{book}{
	author={Pielou, E.},
 	title={An introduction to {Mathematical Ecology}},
	publisher={Wiley-Interscience},
	address={New York},
	date={1974},}

\bib{ostfeld}{article}{
	author={Ostfeld, R.},
	author={Canham, C.},
 	title={Density-dependent processes in meadow voles: an experimental approach},
	date={1995},
	journal={Ecology},
 	volume={76},
 	pages={521\ndash 532},}

\bib{roux}{article}{
	author={Roux, P.},
	author={Shaw, J.},
	author={Chown, S.},
 	title={Ontogenetic shifts in plant interactions vary with environmental severity and affect population structure},
	date={2013},
	journal={New Phytologist},
 	volume={200},
 	pages={241\ndash 250},}

\bib{saez}{article}{
	author={S{\'a}ez, E.},
	author={Gonz{\'a}lez-Olivares, E.},
 	title={Dynamics on a predator--prey model},
	date={1999},
	journal={SIAM Journal on Applied Mathematics},
 	volume={59},
 	pages={1867\ndash 1878},}

\bib{santos}{article}{
	author={Santos, X.},
	author={Cheylan, M.},
 	title={Taxonomic and functional response of a {Mediterranean} reptile assemblage to a repeated fire regime},
	date={2013},
	journal={Biological Conservation},
 	volume={168},
 	pages={90\ndash 98},}

\bib{stephens2}{article}{
	author={Stephens, P.},
	author={Sutherland, W.},
	author={Freckleton, R.},
 	title={What is the {Allee} effect?},
	date={1999},
	journal={Oikos},
 	volume={87},
 	pages={185\ndash 190},}

\bib{tanner}{article}{
	author={Tanner, P.},
	author={Sutherland, J.},
 	title={The stability and the intrinsic growth rates of prey and predator populations},
	date={1975},
	journal={Ecology},
 	volume={21},
 	pages={573\ndash 580},}
        
\bib{turchin}{book}{
	author={Turchin, P.},
 	title={Complex population dynamics: a theoretical/empirical synthesis},
	series={Monographs in population biology},
	publisher={{Princeton} {University} {Press}},
	address={Princeton, N.J.},
	date={2003},
 	volume={35},}

\bib{turchin2}{article}{
	author={Turchin, P.},
	author={Hanski, I.},
 	title={An empirically based model for latitudinal gradient in vole population dynamics},
	date={1997},
	journal={The American Naturalist},
 	volume={149},
 	pages={842\ndash 874},}

\bib{van}{article}{
	author={van Voorn, G.},
	author={Hemerik, L.},
	author={Boer, M.},
	author={Kooi, B.},
 	title={Heteroclinic orbits indicate overexploitation in predator--prey systems with a strong Allee effect},
	date={2007},
	journal={Mathematical Biosciences},
 	volume={209},
 	pages={451\ndash 469},}

\bib{wang7}{article}{
	author={Wang, J.},
	author={Shi, J.},
	author={Wei, J.},
 	title={Predator--prey system with strong {Allee} effect in prey},
	date={2011},
	journal={Journal of Mathematical Biology},
 	volume={62},
 	pages={291\ndash 331},}

\bib{wollkind}{article}{
	author={Wollkind, D.},
	author={Collings, J.},
	author={Logan, J.},
 	title={Metastability in a temperature--dependent model system for predator--prey mite outbreak interactions on fruit trees},
	date={1988},
	journal={Bulletin of Mathematical Biology},
 	volume={50},
 	pages={379\ndash 409},}

\end{thebibliography}
\end{document}